\documentclass[10pt,a4paper]{article}
\usepackage[latin1]{inputenc}
\usepackage{fullpage}
\usepackage{mathrsfs}
\usepackage{amsmath}
\usepackage{amsfonts}
\usepackage{amssymb}
\usepackage{amsthm}
\usepackage{graphicx}
\usepackage{epstopdf}
\usepackage[affil-it]{authblk}
\graphicspath{C:/Users/evanc/OneDrive/Documents/The distributional differintegral}
\DeclareGraphicsExtensions{.eps,.pdf,.png,.jpg,.svg}
\author{Evan Camrud%
	\thanks{Electronic address: \texttt{ecamrud@iastate.edu}; Corresponding author} }
\affil{Mathematics Department, Iowa State University, Ames, IA, 50010}
\title{Fractional operators via analytic interpolation of integer powers}
\newtheorem{Lemma}{Lemma}
\newtheorem{Theorem}{Theorem}
\newtheorem{AlphTheorem}{Theorem}

\theoremstyle{definition}
\newtheorem{Definition}{Definition}
\newtheorem{Remark}{Remark}
\newtheorem{Proposition}{Proposition}
\newtheorem{Corollary}{Corollary}

\newcommand{\eqn}[1]{\begin{displaymath}\begin{split}#1\end{split}\end{displaymath}}
\newcommand{\eqnn}[1]{\begin{equation}\begin{split}#1\end{split}\end{equation}}
\newcommand{\R}{\mathbb{R}}
\newcommand{\Z}{\mathbb{Z}}
\newcommand{\C}{\mathbb{C}}
\newcommand{\N}{\mathbb{N}}

\newcommand{\sinc}{\textnormal   {sinc}}

\newcommand\blfootnote[1]{%
  \begingroup
  \renewcommand\thefootnote{}\footnote{#1}%
  \addtocounter{footnote}{-1}%
  \endgroup
}

\begin{document}

	\maketitle
	

	\textbf{Abstract.} Although the study of functional calculus has already established necessary and sufficient conditions for operators to be fractionalized, this paper aims to use our well-conceived notion of integer powers of operators to construct non-integer powers of operators. In doing so, we not only provide a more intuitive understanding of fractional theories, but also provide a framework for producing new fractional theories. Such interpolations allow one to approximate fractional powers by finite sums of integer powers of operators, and thus may find much use in numerical analysis of fractional PDEs and the time-frequency analysis of the fractional Fourier transform. Further, these results provide an interpolation procedure for the Riemann zeta function.
	
\section{Introduction and motivation}

The transition between discrete and continuous has been an elucidating process in mathematics for several centuries. Whether in the discovery of calculus, or the creation of the real numbers through Dedekind cuts, we have always wondered ``What if we could take infinitely small steps instead?''

Recently, fractional calculus has become a useful tool in studying systems with memory effects \cite{Diethelm}, \cite{Hilfer1}, \cite{Barbosa}; the fractional Fourier transform is often used in time-frequency analysis and digital signal processing \cite{Almeida}, \cite{OzaktasEtAl}, \cite{Soo-Chang}; the theory of pseudo-differential operators has paved the way for modern regularity theory \cite{Hormander}, \cite{Herzog}; the Mellin transform lies at the heart of special function theory and the Riemann hypothesis \cite{Titchmarsh}, and semigroup theory allows one to efficiently solve the abstract Cauchy problem \cite{Hille}, \cite{Balakrishnan}.

Our human minds naturally understand integer powers. Indeed, the mythical execution of Plato's pupil, insisting that the length of the diagonal on a right isoceles triangle was not a rational number, shows just how difficult non-integer powers are to comprehend. Because of this, many of the above-noted fractional powers of operators are constructed in a contrived manner.

Although the study of functional calculus has already established necessary and sufficient conditions for operators to be fractionalized \cite{Komatsu}, \cite{Balakrishnan}, \cite{PrussAndSohr}, this paper aims to use our well-conceived notion of integer powers of operators to construct non-integer powers of operators. In doing so, we not only provide a more intuitive understanding of fractional theories, but also provide a framework for producing new fractional theories.

Let it be noted that fractional powers of linear operators should naturally satisfy
\begin{enumerate}
\item Linearity: $T^\alpha[f+g]=T^\alpha[f]+T^\alpha[g]$
\item Semigroup property: $T^\alpha T^\beta=T^{\alpha+\beta}$
\item Commutivity: $T^\alpha T^\beta=T^\beta T^\alpha$
\item Associativity: $T^\alpha (T^\beta T^\gamma)=(T^\alpha T^\beta)T^\gamma$
\item Integer orders: $\lim_{\alpha\to n}T^\alpha=T^n$ and $\lim_{\alpha\to 0}T^\alpha=T^0=\text{Id}$
\item (Inverses when applicable:) $(T^\alpha)^{-1}=T^{-\alpha}$.
\end{enumerate}
One may quickly notice that if the semigroup property holds for all applicable $\alpha,\beta$, the next four cases follow by the properties of addition. Indeed, the satisfaction of these criteria motivated the construction of fractional powers by interpolation schema. This is because the semigroup property naturally arises in interpolations, as it is inherited from the additivity of integer powers. Linearity follows naturally, as well, from uniform convergence of the series.

By the very nature of this paper, it is necessary to note that we take the principal branch when working with fractional powers (or logarithms).

The manuscript is ordered as follows: in section 2, the two natural interpolation strategies (Newton series, and Shannon interpolation) are introduced and their sufficient conditions for convergence are considered; in section 3, the fractional powers of operators are defined via the two interpolation strategies; in section 4, applications of the fractional powers are given in terms of fractional calculus, Dirichlet series (Mellin transform and the Riemann zeta function), the fractional Fourier transform, and semigroup theory.

\section{Interpolation strategies}

\subsection{Newton series}
Let us first observe the power series, centered at $1$, for the function $x^\alpha$:
\eqn{x^\alpha=\sum_{n=0}^\infty \binom{\alpha}{n}(x-1)^n=\sum_{n=0}^\infty \frac{\Gamma(1+\alpha)}{\Gamma(1+\alpha-n)\Gamma(1+n)}(x-1)^n.}
It is immediately clear that for $\alpha\notin \Z$, the radius of convergence of the power series is $1$. This is due to the presence of the branch cut at $0$.

Supposing that we want to define this series instead on monomials centered at $0$ we utilize the binomial theorem to recover
\eqn{x^\alpha&=\sum_{n=0}^\infty \binom{\alpha}{n}\bigg[\sum_{k=0}^n\binom{n}{k}(-1)^{n-k} x^k\bigg]=\sum_{n=0}^\infty \bigg[\sum_{k=0}^n(-1)^k\binom{n}{k} x^k\bigg](-1)^n\binom{\alpha}{n}.}

Written as such, the following definition seems straightforward.

\begin{Definition}
The \textbf{$n^{th}$ Pochhammer-Newton polynomial} is given by
\eqn{P_n(\alpha)=(-1)^n\binom{\alpha}{n}=\frac{(-1)^n}{\Gamma(n+1)}\prod_{m=0}^{n-1}(\alpha-m)=\frac{(-1)^n \Gamma(\alpha+1)}{\Gamma(\alpha-n+1)\Gamma(n+1)}.}
\end{Definition}

This now gives the identity
\eqnn{x^\alpha=\sum_{n=0}^\infty\bigg[\sum_{k=0}^n P_k(n)x^k\bigg]P_n(\alpha)}
convergent for all $x$ such that $|x-1|<1$.

Indeed, this is a special case of the so-called \emph{Newton series} of polynomial interpolation.

\begin{Definition}
The \textbf{Newton series} for a function $f$ is given by
\eqn{f(x)=\sum_{n=0}^\infty (-1)^n \Delta^n[f](x_0) P_n(x-x_0)}
where $\Delta^n[\cdot]$ is the finite-difference operator. However, we expand the operator and define it as
\eqn{f(x)&=\sum_{n=0}^\infty \bigg[\sum_{k=0}^n P_k(n)f(k-x_0)\bigg]P_n(x-x_0).}
\end{Definition}

\begin{Remark}
For $x_0=0$ we see the \emph{interpolation of integer values}:
\eqnn{f(x)&=\sum_{n=0}^\infty \bigg[\sum_{k=0}^n P_k(n)f(k)\bigg]P_n(x).}
such that a function is determined by its values at the non-negative integers.
\end{Remark}

While the Newton series is known to converge for functions of exponential type slightly greater than $\log 2$, as seen in \cite{BoasBuck}, \cite{Klopfenstein}, this is the natural cut-off for exponential type. Thus, as was motivated heuristically, we consider the Newton series for $f(x)=x^\alpha$ converging for $|x-1|<1$.

\subsection{Shannon interpolation}

\begin{AlphTheorem}[Shannon sampling theorem \cite{Whittaker}]\label{T:Shannon}
	Suppose $f\in L^2(\R)$ such that $\textnormal{supp}[\hat{f}]\subseteq[-\frac{1}{2},\frac{1}{2}]$. Then
	\eqnn{f(x)=\sum_{n\in \Z}\sinc(x-n)f(n)}
	where $\hat{f}$ is the Fourier transform of $f$ and $\sinc(x)=\frac{\sin\pi x
	}{\pi x}$. Further, the series converges absolutely and uniformly.
\end{AlphTheorem}

Since we care about fractional powers, it is important to ask the question, does
\eqn{x^\alpha=\sum_{n\in \Z}\sinc(\alpha-n)x^n?}
The answer seems as though it would be ``no'', as $f_x(\alpha)=x^\alpha=e^{\alpha\log x}$ is certainly not in $L^2(\R)$. However, the Fourier transform is well-defined in the sense of tempered distributions as
\eqn{\mathcal{F}[e^{\alpha\log x}]=\int_\R e^{\alpha\log x}e^{-2\pi i\alpha\beta}d\alpha=\int_\R e^{-2\pi i\alpha(\beta-\frac{\log x}{2\pi i})}d\alpha=\delta\Big(\beta-\frac{\log x}{2\pi i}\Big)}
where for $\frac{\log x}{2\pi i}\in[-\frac{1}{2},\frac{1}{2}]\iff \log x\in[-\pi i,\pi i]\iff |x|=1$ the second condition of the Shannon sampling theorem is satisfied.

\begin{Proposition}\label{P:ShannonUnit}
	\eqnn{x^\alpha=\sum_{n\in \Z}\sinc(\alpha-n)x^n}
	if and only if $|x|=1$. The convergence of the series happens pointwise in $\alpha$.
\end{Proposition}

\begin{proof}
	($\sim\Longleftarrow$) Suppose $|x|\neq 1$. Then either $|x|>1$ or $|x^{-1}|>1$. Without loss of generality, let $|x|>1$. As such if one considers
	\eqn{\sum_{n\in \Z}\sinc(\alpha-n)x^n=\sinc(\alpha)+\sum_{n\in\N} \sinc(\alpha+n)x^{-n}+\sum_{n\in\N} \sinc(\alpha-n)x^n}
	then the first sum on the right-hand side is clearly convergent as it is bounded by a geometric series. Hence, the left-hand side will converge if and only if
	\eqn{\sum_{n=0}^\infty \sinc(\alpha-n)x^n}
	converges. This cannot be the case, however, as $\lim_{n\to\infty}|\sinc(\alpha-n)x^n|=\infty\neq 0$. ($\qed\sim\Longleftarrow$)\\
	
	($\Longleftarrow$) There are two cases.

	\textbf{Case 1:} $|x|=1$ and $x\neq-1$.

	We observe that $\sin(\pi(\alpha-n))=\sin(\pi\alpha-\pi n)=(-1)^n\sin(\pi\alpha)$ and hence we have that
	\eqn{\sum_{n\in\Z}\sinc(\alpha-n)x^n&=\sum_{n\in \Z}\frac{\sin(\pi(\alpha-n))}{\pi(\alpha-n)}x^n=\frac{\sin(\pi\alpha)}{\pi}\sum_{n\in \Z}\frac{(-x)^n}{\alpha-n}\\
	&=\frac{\sin(\pi\alpha)}{\pi}\bigg[\frac{1}{\alpha}+\sum_{n\in\N}\frac{(-x)^n}{\alpha-n}+\sum_{n\in\N}\frac{(-x)^{-n}}{\alpha+n}\bigg].}
	We now observe by Dirichlet's test, with $a_n=\frac{1}{n-\alpha}$ and $b_n=(-x)^{n}$ we have 
	\begin{enumerate}
		\item for $n$ large enough $a_{n+1}\leq a_n$
		\item $\lim_{n\to\infty}a_n=0$
		\item $\Big|\sum_{n=0}^N b_n\Big|=\Big|\frac{1-(-x)^n}{1-(-x)}\Big|\leq \frac{2}{|1+x|}<\infty.$
	\end{enumerate}
	Hence the first series converges. Similarly, with $a_n=\frac{1}{n+\alpha}$ and $b_n=(-x)^{-n}$ that 
	\begin{enumerate}
		\item for $n$ large enough $a_{n+1}\leq a_n$
		\item $\lim_{n\to\infty}a_n=0$
		\item $\Big|\sum_{n=0}^N b_n\Big|=\Big|\frac{1-(-x)^{-n}}{1-(-x)^{-1}}\Big|\leq \frac{2}{|1+x^{-1}|}<\infty.$
	\end{enumerate}
	Hence the second series converges.

	\textbf{Case 2:} $x=-1$.

	\eqn{\sum_{n\in\Z}\sinc(\alpha-n)x^n&=\sum_{n\in \Z}\frac{\sin(\pi(\alpha-n))}{\pi(\alpha-n)}x^n=\frac{\sin(\pi\alpha)}{\pi}\sum_{n\in \Z}\frac{(-x)^n}{\alpha-n}\\
	&=\frac{\sin(\pi\alpha)}{\pi \alpha}+\frac{\sin(\pi\alpha)}{\pi}\sum_{n\in\N}\bigg[\frac{(-x)^n}{\alpha-n}+\frac{(-x)^{-n}}{\alpha+n}\bigg]\\
	&=\frac{\sin(\pi\alpha)}{\pi \alpha}+\frac{\sin(\pi\alpha)}{\pi}\sum_{n\in\N}\bigg[\frac{(\alpha+n)(-x)^n+(\alpha-n)(-x)^{-n}}{\alpha^2-n^2}\bigg]\\
	&=\frac{\sin(\pi\alpha)}{\pi \alpha}+\frac{\sin(\pi\alpha)}{\pi}\sum_{n\in\N}\frac{\alpha\big[(-x)^n+(-x)^{-n}\big]}{\alpha^2-n^2}+\frac{\sin(\pi\alpha)}{\pi}\sum_{n\in\N}\frac{n\big[(-x)^n-(-x)^{-n}\big]}{\alpha^2-n^2}.}

	However, the first series in the last line is on the order of $\sum_{n\in \N}\frac{1}{n^2}<\infty$ and the second series reduces to $\sum_{n\in\N}\frac{n\cdot 0}{\alpha^2-n^2}=0$ since $x=-1$. Therefore the series converges. ($\qed\Longleftarrow$)
\end{proof}

\begin{Proposition}\label{P:ShannonPeriod}
If $\{c_n\}_{n\in\Z}$ is periodic with period $N$. Then
\eqnn{\sum_{n\in\Z}\sinc(x-n)c_n=\begin{cases}
\frac{1}{N}\sum_{n=0}^{N-1} \sin\big(\pi (x-n)\big)\cot\big(\frac{\pi}{N} (x-n)\big)c_n&\text{if }N\text{ is even,}\\
\frac{1}{N}\sum_{n=0}^{N-1} \sin\big(\pi (x-n)\big)\csc\big(\frac{\pi}{N} (x-n)\big)c_n&\text{if }N\text{ is odd.}
\end{cases}}
\end{Proposition}

\begin{proof}
\eqn{\sum_{n\in\Z}\sinc(x-n)c_n=\sum_{n=0}^{N-1} \bigg[\sum_{m\in\Z} \sinc\big((x-n)-Nm)\big)\bigg]c_n.}
Further we have that
\eqn{\sum_{m\in\Z} \sinc(x-Nm)=\sum_{m\in\Z} \sinc(x-m)d_n}
where $d_{kN}=1$ for all $k\in\Z$ and is otherwise zero. Since
\eqn{f(x)=\begin{cases}
\frac{1}{N}\sin\big(\pi x\big)\cot\big(\frac{\pi}{N} x\big)&\text{if }N\text{ is even,}\\
\frac{1}{N}\sin\big(\pi x\big)\csc\big(\frac{\pi}{N} x\big)&\text{if }N\text{ is odd}
\end{cases}}
has the property that $f(kN)=1$ for all $k\in\Z$ and is otherwise zero at the integers, $f(n)=d_n$ for all $n\in\Z$, and it is readily checked that $\textnormal{supp}[\hat{f}]\subseteq[-\frac{1}{2},\frac{1}{2}]$ in the sense of tempered distributions. Therefore
\eqn{f(x-n)=\sum_{m\in\Z} \sinc\big((x-n)-m)\big)f(m-n)=\sum_{m\in\Z} \sinc\big((x-n)-m)\big)d_n=\sum_{m\in\Z} \sinc\big((x-n)-Nm)\big).}

Hence
\eqn{\sum_{n\in\Z}\sinc(x-n)c_n&=\sum_{n=0}^{N-1} \bigg[\sum_{m\in\Z} \sinc\big((x-n)-Nm)\big)\bigg]c_n=\sum_{n=0}^{N-1} f(x-n)c_n\\
&=\begin{cases}
\frac{1}{N}\sum_{n=0}^{N-1} \sin\big(\pi (x-n)\big)\cot\big(\frac{\pi}{N} (x-n)\big)c_n&\text{if }N\text{ is even,}\\
\frac{1}{N}\sum_{n=0}^{N-1} \sin\big(\pi (x-n)\big)\csc\big(\frac{\pi}{N} (x-n)\big)c_n&\text{if }N\text{ is odd.}
\end{cases}}

\end{proof}

\section{Interpolation of operator powers}

The functional calculus requires that for $T$ (a bounded linear operator) and $f$ (a function), the domain of $f$ must contain the spectrum of $T$ to well-define $f(T)$. Our two interpolation strategies give $f(x)=x^\alpha$ on the domains $|x-1|<1$ and $|x|=1$ respectively. Hence, if we have any hope for defining $T^\alpha$, we must require that $|\sigma(T)-1|<1$ or $|\sigma(T)|=1$. This naturally leads to two cases.

We define $\C^+=\{z\in\C:\text{Re}[z]>0\}$.

\subsection{$\sigma(T)\subset \overline{B(z,|z|)}$ for some $z\in \C$}

\begin{Theorem}\label{T:NewtonOperator1}
For $T$ a bounded linear operator such that $\sigma(T)\subsetneq {B(z,|z|)}$, then there exists $\rho\in \C$ such that $\frac{\sigma(T)}{\rho}=\sigma \big(\frac{T}{\rho}\big)\subset B(1,1)$. Then
\eqn{\Big(\frac{T}{\rho}\Big)^\alpha=\sum_{n=0}^\infty\bigg[\sum_{k=0}^n P_k(n)\Big(\frac{T}{\rho}\Big)^k\bigg]P_n(\alpha).}
As such
\eqnn{T^\alpha =\sum_{n=0}^\infty\bigg[\sum_{k=0}^n P_k(n) \rho^{\alpha-k}T^k\bigg]P_n(\alpha).}
\end{Theorem}

\begin{proof}
Since $\sigma\big(\frac{T}{\rho}\big)\subset B(1,1)$, we let $\Gamma\subset B(1,1)$ be a closed curve enclosing $\sigma\big(\frac{T}{\rho}\big)$. Then we have from the holomorphic functional calculus
\eqn{\Big(\frac{T}{\rho}\Big)^\alpha&=\frac{1}{2\pi i}\oint_\Gamma \lambda^\alpha \Big(\lambda-\frac{T}{\rho}\Big)^{-1}d\lambda=\frac{1}{2\pi i}\oint_\Gamma \sum_{n=0}^\infty\bigg[\sum_{k=0}^n P_k(n) \lambda^k \Big(\lambda-\frac{T}{\rho}\Big)^{-1}\bigg]P_n(\alpha)d\lambda\\
&=\sum_{n=0}^\infty\bigg[\sum_{k=0}^n P_k(n) \bigg(\frac{1}{2\pi i}\oint_\Gamma \lambda^k \Big(\lambda-\frac{T}{\rho}\Big)^{-1} d\lambda\bigg)\bigg]P_n(\alpha)=\sum_{n=0}^\infty\bigg[\sum_{k=0}^n P_k(n)\Big(\frac{T}{\rho}\Big)^k\bigg]P_n(\alpha)}
where the sum and the integral are exchanged with the understanding that the integral is a Bochner integral, in which a limit of Riemann sums converges uniformly in the strong operator topology.
\end{proof}

\begin{Theorem}\label{T:NewtonOperator2}
For $T$ a bounded linear operator such that $\sigma(T)\subset \overline{B(z,|z|)}$, then there exists $\rho\in \C$ such that $\frac{\sigma(T)}{\rho}=\sigma \big(\frac{T}{\rho}\big)\setminus\{0\}\subset B(1,1)$. Then, for all $\alpha\in\C^+$,
\eqn{\Big(\frac{T}{\rho}\Big)^\alpha=\sum_{n=0}^\infty\bigg[\sum_{k=0}^n P_k(n)\Big(\frac{T}{\rho}\Big)^k\bigg]P_n(\alpha).}
As such, for all $\alpha\in\C^+$,
\eqnn{T^\alpha =\sum_{n=0}^\infty\bigg[\sum_{k=0}^n P_k(n) \rho^{\alpha-k}T^k\bigg]P_n(\alpha).}
\end{Theorem}

\begin{proof}
We note that since $\sigma\big(\frac{T}{\rho}\big)\setminus\{0\}\subsetneq B(1,1)$ then there is an $\epsilon>0$ such that $\sigma\big(\frac{T}{\rho}+\epsilon\big)\subsetneq B(1,1)$ Then theorem (\ref{T:NewtonOperator1}) applies and we have that
\eqn{\Big(\frac{T}{\rho}+\epsilon\Big)^\alpha=\sum_{n=0}^\infty\bigg[\sum_{k=0}^n P_k(n)\Big(\frac{T}{\rho}+\epsilon\Big)^k\bigg]P_n(\alpha).}
However, this series is uniformly convergent in the operator norm, and as such we have that, for all $\alpha\in\C^+$,
\eqn{\Big(\frac{T}{\rho}\Big)^\alpha&=\lim_{\epsilon\to 0^+}\Big(\frac{T}{\rho}+\epsilon\Big)^\alpha=\lim_{\epsilon\to 0^+}\sum_{n=0}^\infty\bigg[\sum_{k=0}^n P_k(n)\Big(\frac{T}{\rho}+\epsilon\Big)^k\bigg]P_n(\alpha)\\
	&=\sum_{n=0}^\infty\bigg[\sum_{k=0}^n P_k(n)\lim_{\epsilon\to 0^+}\Big(\frac{T}{\rho}+\epsilon\Big)^k\bigg]P_n(\alpha)=\sum_{n=0}^\infty\bigg[\sum_{k=0}^n P_k(n)\Big(\frac{T}{\rho}\Big)^k\bigg]P_n(\alpha).}
\end{proof}

\subsection{$|\lambda|=1$ for all $\lambda\in\sigma(T)$}

\begin{Theorem}\label{T:ShannonOperator}
Suppose $T$ is unitary. Then
\eqnn{T^\alpha=\sum_{n\in\Z}\sinc(\alpha-n)T^n.}
\end{Theorem}

\begin{proof}
Since $T$ is unitary, it is a normal operator, and as such there exists a continuous functional calculus on $T$. Hence for any continuous function $f$, $f(T)$ is defined if and only if $f(\sigma(T))$ is defined. By proposition (\ref{P:ShannonUnit}), since $|\lambda|=1$ for all $\lambda\in\sigma(T)$, we have that
\eqn{\sigma(T)^\alpha=\sum_{n\in \Z}\sinc(\alpha-n)\sigma(T)^n}
and as such
\eqn{T^\alpha=\sum_{n\in\Z}\sinc(\alpha-n)T^n.}
\end{proof}

\begin{Corollary}\label{C:PreserveUnitarity}
If $T$ is unitary, then $T^\alpha$ is unitary for all $\alpha\in\R$.
\end{Corollary}

\begin{proof}
\eqn{T^{-\alpha}=\sum_{n\in \Z}\sinc(-\alpha-n)T^n=\sum_{n\in \Z}\sinc(\alpha+n)T^n=\sum_{n\in \Z}\sinc(\alpha-n)T^{-n}=\sum_{n\in \Z}\sinc(\alpha-n)(T^*)^n=(T^*)^\alpha.}
\end{proof}

\begin{Theorem}\label{T:PeriodicOperator}
If $T^N=T$ for some $N\in \N$. Then
\eqn{T^\alpha=\begin{cases}
\frac{1}{N}\sum_{n=0}^{N-1} \sin\big(\pi (\alpha-n)\big)\cot\big(\frac{\pi}{N} (\alpha-n)\big)T^n&\text{if }N\text{ is even,}\\
\frac{1}{N}\sum_{n=0}^{N-1} \sin\big(\pi (\alpha-n)\big)\csc\big(\frac{\pi}{N} (\alpha-n)\big)T^n&\text{if }N\text{ is odd.}
\end{cases}}
\end{Theorem}

\begin{proof}
This is a direct result of proposition (\ref{P:ShannonPeriod}).
\end{proof}

\section{Application}

\subsection{Fractional calculus}

\subsubsection{Fractional integration}

The integration operator $J:L^1(a,b)\to L^1(a,b)$, for $a\in[-\infty,\infty),b\in(a,\infty]$, defined via
	\eqn{J[f](x)=\int_a^x f(y)dy}
has the property that
\eqn{(J-\lambda)[f](x)&=\int_a^x f(y)dy-\lambda f(x)=g(x)\iff\\
f(x)=(J-\lambda)^{-1}[g](x)&=-\lambda^{-1}e^{\lambda^{-1}(x-a)}g(a)-\lambda^{-1}\int_a^x e^{\lambda^{-1} (x-y)}g'(y)dy}
where $g'$ is considered as the \emph{weak} derivative of $g$, and the expression only fails when $\lambda=0$, in which case we have
\eqn{J[f](x)=\int_a^x f(y)dy=g(x)\iff f(y)=J^{-1}[g](x)=g'(x)}
where this expression is satisfied if and only if $g$ is differentiable and $g(a)=0$.

Thus in the case where $a\neq -\infty$ we have that $\sigma(J)=\{0\}$, while in the case where $a=-\infty$ we have that $\sigma(J)=\overline{\C^+}$ as $J[e^{\lambda x}](x)=\lambda^{-1} e^{\lambda x}$ for all $\lambda\in\C^+$.

\begin{Proposition}\label{P:RLBounded}
	$J:L^1(a,b)\to L^1(a,b)$ for $a\neq-\infty$ has fractional power, defined for all $\alpha\in\C^+$,
	\eqnn{J^\alpha=\sum_{n=0}^\infty\bigg[\sum_{k=0}^n P_k(n)J^k\bigg]P_n(\alpha).}
\end{Proposition}

\begin{proof}
We have already shown that $\sigma(J)=\{0\}$. Thus by theorem (\ref{T:NewtonOperator2}) above, we have that, for all $\alpha\in\C^+$,
\eqn{J^\alpha=\sum_{n=0}^\infty\bigg[\sum_{k=0}^n P_k(n)J^k\bigg]P_n(\alpha).}
\end{proof}

Let $\mathcal{T}_\epsilon(-\infty,b)=\overline{\text{span}}\{f\in L^1(-\infty,b):f(x)=\mathcal{O}(e^{\epsilon x})\text{ as }x\to -\infty\}$ for any $b\in (-\infty,\infty)$.

\begin{Proposition}\label{P:RLType}
	The integration operator $J:\mathcal{T}
	_\epsilon\to\mathcal{T}_\epsilon$
	has fractional power
	\eqn{J^\alpha=\sum_{n=0}^\infty\bigg[\sum_{k=0}^n P_k(n)\rho^{\alpha-k}J^k\bigg]P_n(\alpha)}
	for any $\rho>\frac{2}{\epsilon}$ and $\alpha\in \C^+$.
\end{Proposition}

\begin{proof}
A simple calculation gives that
\eqn{J[e^{(\epsilon+i\lambda) x}](x)=(\epsilon+i\lambda)^{-1}e^{(\epsilon+i\lambda) x}}
and hence $\sigma(J)\subseteq \overline{B\big(\frac{\epsilon}{2},\frac{\epsilon}{2}\big)}$. Thus we satisfy theorem (\ref{T:NewtonOperator2}) for any $\rho>\epsilon^{-1}$.
\end{proof}

\begin{Corollary}\label{C:RLUnbounded}
The integration operator $J:L^1(-\infty,b)\to L^1(-\infty,b)$
has fractional power
\eqn{J^\alpha[f]=\lim_{\rho\to\infty}\sum_{n=0}^\infty\bigg[\sum_{k=0}^n P_k(n)\rho^{\alpha-k}J^k[f]\bigg]P_n(\alpha)}
for $\alpha\in \C^+$. Hence this convergence happens pointwise in the domain of the operator (or equivalently, in the $L^1(-\infty,b)$ norm).
\end{Corollary}

\begin{proof}
For each $f\in L^1(-\infty,b)$, we have that $f(x)=\mathcal{O}(e^{\epsilon x})$ as $x\to-\infty$ for some $\epsilon>0$. The rest follows by proposition (\ref{P:RLType}).
\end{proof}

\begin{Definition}
The \textbf{Riemann-Liouville} fractional integral $\widetilde{J}^\alpha:L^1(a,b)\to L^1(a,b)$, for $a\in[-\infty,\infty)$ and $b\in(a,\infty]$, is given by
\eqn{\widetilde{J}^\alpha[f](x)=\frac{1}{\Gamma(\alpha)}\int_a^x f(y)(x-y)^{\alpha-1}dy.}
\end{Definition}

\begin{Proposition}
$J^\alpha=\widetilde{J}^\alpha$.
\end{Proposition}

\begin{proof}
By definition of $J^\alpha$, and following the integral found in \cite{Komatsu}, we have that
\eqn{J^\alpha[f](x)&=\frac{1}{2\pi i}\oint_\Gamma \lambda^\alpha (\lambda-J)^{-1}d\lambda=-\frac{\sin\pi\alpha}{\pi}\int_0^\infty \lambda^\alpha (\lambda+J)^{-1}d\lambda\\
&=\frac{\sin\pi\alpha}{\pi}\int_0^\infty\bigg[\lambda^{\alpha-1}e^{-\lambda^{-1}(x-a)}f(a)+\lambda^{\alpha-1}\int_a^x e^{-\lambda^{-1}(x-y)}f'(y)dy \bigg]d\lambda\\
&=\frac{\sin\pi\alpha}{\pi}f(a)\int_0^\infty\lambda^{1-\alpha}e^{-\lambda(x-a)}d\lambda+\frac{\sin\pi\alpha}{\pi}\int_a^x f'(y)\bigg[\int_0^\infty\lambda^{1-\alpha} e^{-\lambda(x-y)}d\lambda\bigg]dy\\
&=\frac{\sin\pi\alpha}{\pi}\Gamma(2-\alpha)f(a)(x-a)^{\alpha-2}+\frac{\sin\pi\alpha}{\pi}\Gamma(2-\alpha)\int_a^x f'(y)(x-y)^{\alpha-2} dy\\
&=\frac{1}{\Gamma(\alpha-1)}f(a)(x-a)^{\alpha-2}+\frac{1}{\Gamma(\alpha-1)}\int_a^x f'(y)(x-y)^{\alpha-2} dy\\
&=\frac{1}{\Gamma(\alpha)}\int_a^x f(y)(x-y)^{\alpha-1} dy\\
&=\widetilde{J}^\alpha[f](x).}

\end{proof}

\subsubsection{Fractional differentiation}

\begin{Lemma}\label{L:DTrig1}
	The fractional derivative of $\sin(x),\cos(x)$ is
	\eqnn{D^\alpha[\sin(x)](x)&=\sin\Big(x+\frac{\pi}{2}\alpha\Big)\\
		D^\alpha[\cos(x)](x)&=\cos\Big(x+\frac{\pi}{2}\alpha\Big).}
\end{Lemma}

\begin{proof}
We note that $D^2[\sin(x)](x)=-\sin(x)$ and $D^4[\sin(x)](x)=\sin(x)$. Hence by theorem (\ref{T:PeriodicOperator}) we have that
\eqn{D^\alpha[\sin(x)](x)&=\frac{\sin(\pi\alpha)}{4}\bigg[\cot\Big(\frac{\pi}{4}\alpha\Big)+\tan\Big(\frac{\pi}{4}\alpha\Big)\bigg]\sin(x)-\frac{\sin(\pi\alpha)}{4}\bigg[\cot\Big(\frac{\pi}{4}(\alpha-1)\Big)+\tan\Big(\frac{\pi}{4}(\alpha-1)\Big)\bigg]\cos(x)\\
&=\frac{\sin(\pi\alpha)}{4}\bigg[\frac{2}{\sin\big(\frac{\pi}{2}\alpha\big)}\bigg]\sin(x)+\frac{\sin(\pi\alpha)}{4}\bigg[\frac{2}{\cos\big(\frac{\pi}{2}\alpha\big)}\bigg]\cos(x)=\cos\big(\frac{\pi}{2}\alpha\big)\sin(x)+\sin\big(\frac{\pi}{2}\alpha\big)\cos(x)\\
&=\frac{\sin(x+\frac{\pi}{2}\alpha)-\sin(x-\frac{\pi}{2}\alpha)}{2}+\frac{\sin(x+\frac{\pi}{2}\alpha)+\sin(x-\frac{\pi}{2}\alpha)}{2}=\sin\Big(x+\frac{\pi}{2}\alpha\Big)}
under trigonometric identities. A simple translation gives the result for cosine.
\end{proof}

\begin{Lemma}\label{L:DTrig2}
	The fractional derivative of $\sin(\lambda x),\cos(\lambda x)$ is
	\eqn{D^\alpha[\sin(\lambda x)]&=\lambda^\alpha \sin\Big(\lambda x+\frac{\pi}{2}\alpha\Big)\\
	D^\alpha[\cos(\lambda x)]&=\lambda^\alpha \cos\Big(\lambda x+\frac{\pi}{2}\alpha\Big).}
\end{Lemma}

\begin{proof}
	The operator $\lambda^{-1}D$ has order four on these functions. The proof then follows the same as for the previous lemma (\ref{L:DTrig1}).
\end{proof}

\begin{Proposition}\label{P:DFourierSeries}
	The fractional derivative of $f\in L^2(a,b)$ is given by
	\eqn{D^\alpha[f](x)=\sum_{n\in\Z}\hat{f}(n) \Big(\frac{2\pi n}{b-a}\Big)^\alpha e^{i(\frac{2\pi n}{b-a} x+\frac{\pi}{2}\alpha)}}
	where
	\eqn{f(x)=\sum_{n\in\Z}\hat{f}(n) e^{i\frac{2\pi n}{b-a} x}}
	is the Fourier series expansion of $f$.
\end{Proposition}

\begin{proof}
	It is enough to show that $D^\alpha[e^{i\lambda x}]=\lambda^\alpha e^{i(\lambda x+\frac{\pi}{2}\alpha)}$ for all $\lambda\in\R$. This is a direct result of lemma (\ref{L:DTrig2}) above as
	\eqn{D^\alpha[e^{i\lambda x}]&=D^\alpha[\cos(\lambda x)+i\sin(\lambda x)]=D^\alpha[\cos(\lambda x)]+i D^\alpha[\sin(\lambda x)]\\
		&=\lambda^\alpha\Big[ \cos\Big(\lambda x+\frac{\pi}{2}\alpha\Big)+i\sin\Big(\lambda x+\frac{\pi}{2}\alpha\Big)\Big]=\lambda^\alpha e^{i(\lambda x+\frac{\pi}{2}\alpha)}.}
\end{proof}

\begin{Proposition}\label{P:DFourierTransform}
	The fractional derivative of $f\in L^2(\R)$ is given by
	\eqn{D^\alpha[f](x)=(2\pi i)^\alpha\int_\R\hat{f}(y) y^\alpha e^{2\pi ixy}dy}
	where $\hat{f}$ is the Fourier transform of $f$.
\end{Proposition}

\begin{proof}
	This follows in the exact same manner as the proof to the previous proposition.
\end{proof}

\subsection{Dirichlet series}

\begin{Proposition}\label{P:MellinTransform}
	Suppose $f(x)=\sum_{n=1}^\infty a_n e^{-nx}$ for some $\{a_n\}_{n=1}^\infty\in\ell^\infty$. Then
	\eqn{\mathcal{M}[f](s)=\sum_{n=0}^\infty\bigg[\sum_{k=0}^n P_k(n)\frac{\Gamma(s)}{\Gamma(k)}\mathcal{M}[f](k)\bigg]P_n(s)}
	where $\mathcal{M}[f](s)=\int_0^\infty f(x)x^{s-1}dx$ is the Mellin transform of $f$.
\end{Proposition}

\begin{proof}
Since
\eqn{f(x)=\sum_{n=1}^\infty a_n e^{-nx},}
then clearly $\check{f}\in\mathcal{T}_1$ from proposition (\ref{P:RLType}) above, where $\check{f}(x)=f(-x)$. This implies that
\eqnn{\label{E:Jalpha1}J^\alpha[\check{f}](x)=\frac{1}{\Gamma(s)}\int_{-\infty}^x \check{f}(y)(x-y)^{s-1}dy}
as well as
\eqnn{\label{E:Jalpha2}J^\alpha[\check{f}](x)=\sum_{n=0}^\infty\bigg[\sum_{k=0}^n P_k(n)J^k[\check{f}](x)\bigg]P_n(\alpha)=\sum_{n=0}^\infty\bigg[\sum_{k=0}^n P_k(n)\frac{1}{\Gamma(k)}\int_{-\infty}^x \check{f}(y)(x-y)^{k-1}dy\bigg]P_n(s).}

Evaluating equations (\ref{E:Jalpha1}) and (\ref{E:Jalpha2}) above at $x=0$ and observing the integral under the change-of-variable $y\mapsto -y$ we recover
\eqn{\frac{1}{\Gamma(s)}\int_0^\infty f(y)y^{s-1}dy=\sum_{n=0}^\infty\bigg[\sum_{k=0}^n P_k(n)\frac{1}{\Gamma(k)}\int_{0}^\infty f(y)y^{k-1}dy\bigg]P_n(s)}
which is equivalently
\eqn{\mathcal{M}[f](s)=\sum_{n=0}^\infty\bigg[\sum_{k=0}^n P_k(n)\frac{\Gamma(s)}{\Gamma(k)}\mathcal{M}[f](k)\bigg]P_n(s).}
\end{proof}

\begin{Lemma}\label{L:Dirichlet}
For any Dirichlet series $g(s)=\sum_{n=1}^\infty \frac{a_n}{n^s}$
\eqn{g(s)=\sum_{n=0}^\infty\bigg[\sum_{k=0}^n P_k(n)g(k)\bigg]P_n(s)}
\end{Lemma}

\begin{proof}
	Note $g$ is of the form
	$g(s)=\frac{\mathcal{M}[f](s)}{\Gamma(s)}$ for $f(x)=\sum_{n=1}^\infty a_n e^{-n x}$. Hence
	\eqn{g(s)=\sum_{n=0}^\infty\bigg[\sum_{k=0}^n P_k(n)g(k)\bigg]P_n(s)}
	so long as $\{g(k)\}_{k=0}^\infty$ is defined.
\end{proof}

\begin{Lemma}\label{L:EtaFunction}
	The Dirichlet eta function, given by $\eta(s)=\sum_{n=1}^\infty \frac{(-1)^n}{n^s}$, has the property that
	\eqn{\eta(s)=\sum_{n=0}^\infty\bigg[\sum_{k=0}^n P_k(n)\eta(k)\bigg]P_n(s).}
\end{Lemma}

\begin{proof}
This is a direct result of lemma (\ref{L:Dirichlet}).
\end{proof}

\begin{Theorem}\label{T:ZetaFunction}
	The Riemann zeta function has the property that
	\eqn{\zeta(s)=\sum_{n=0}^\infty\bigg[\sum_{k=0}^n P_k(n)\frac{1-2^{1-k}}{1-2^{1-s}}\zeta(k)\bigg]P_n(s).}
\end{Theorem}

\begin{proof}
	Since $\eta(s)=(1-2^{1-s})\zeta(s)$, this is a direct result of lemma (\ref{L:Dirichlet}) above.
\end{proof}

\begin{Corollary}\label{C:RecipZetaFunction}
	The reciprocal Riemann zeta function has the property that
	\eqn{\frac{1}{\zeta(s)}=\sum_{n=0}^\infty\bigg[\sum_{k=0}^n P_k(n)\frac{1}{\zeta(k)}\bigg]P_n(s).}
\end{Corollary}

\begin{proof}
	Since $\frac{1}{\zeta(s)}=\sum_{n=1}^\infty \frac{\mu(n)}{n^s}$, where $\mu(n)$ is the M\"obius function, this proof is the same as the first part of lemma (\ref{L:Dirichlet}) above.
\end{proof}

\begin{Corollary}\label{C:ZetaFunction}
	The Riemann zeta function has the property that
	\eqn{\zeta(s+1+\epsilon)=\sum_{n=0}^\infty\bigg[\sum_{k=0}^n P_k(n)\zeta(k+1+\epsilon)\bigg]P_n(s).}
	for all $\epsilon>0$.
\end{Corollary}

\begin{proof}
	Since $\zeta(s)=\sum_{n=1}^\infty \frac{1}{n^s}$, the proof is equivalent to that of lemma (\ref{L:Dirichlet}).
\end{proof}

\subsection{Fractional Fourier transform}

\begin{Proposition}[Alternate fractional Fourier transform]\label{P:AltFrFT}
\eqnn{\mathcal{F}^\alpha&=\frac{\sin(\pi\alpha)}{4}\bigg[\cot\Big(\frac{\pi}{4}\alpha\Big)\mathcal{F}^0-\cot\Big(\frac{\pi}{4}(\alpha-1)\Big)\mathcal{F}-\tan\Big(\frac{\pi}{4}\alpha\Big)\mathcal{F}^2+\tan\Big(\frac{\pi}{4}(\alpha-1)\Big)\mathcal{F}^3\bigg]\\
\mathcal{F}^\alpha[f](x)&=\frac{\sin(\pi\alpha)}{4}\bigg[\cot\Big(\frac{\pi}{4}\alpha\Big)f(x)-\cot\Big(\frac{\pi}{4}(\alpha-1)\Big)\hat{f}(x)-\tan\Big(\frac{\pi}{4}\alpha\Big)f(-x)+\tan\Big(\frac{\pi}{4}(\alpha-1)\Big)\hat{f}(-x)\bigg].}
In this manner, the fractional Fourier transform can be completely determined by knowing only the function and its Fourier transform.
\end{Proposition}

\begin{proof}
This is a direct result of theorem (\ref{T:PeriodicOperator}).
\end{proof}

We note that this definition of a fractional Fourier transform does not agree with that in the literature, namely that
\begin{Definition}[Literature fractional Fourier transform]
\eqn{\mathcal{F}_\alpha[f](u)=\sqrt{1-i\cot(\alpha)}e^{i\pi\cot(\alpha)u^2}\int_{\R}e^{-i2\pi\big(\csc(\alpha)ux-\frac{\cot(\alpha)}{2}x^2\big)}f(x)dx}
where $\alpha\in[0,2\pi)$ is considered as a rotation angle in the two-dimensional time-frequency domain.
\end{Definition}

It must be noted, however, that the alternate fractional Fourier transform has some advantages to that in the literature. One of these comes in ease of calculation, as shown above, since knowing only the function and its Fourier transform is enough to calculate the fractional Fourier transform. Further, it retains the properties of additivity, linearity, integer orders, inverses, commutativity, associativity, unitarity, and time reversal (\emph{i.e.} it commutes with the involution).

Below you will find a figure comparing the two definitions of the fractional Fourier transform on the function $f(x)=\chi_{[-1/2,1/2]}(x)$ such that $\hat{f}(x)=\sinc(x)$.

\begin{figure}[h!]
  \centering
  \begin{minipage}[b]{0.19\textwidth}
    \includegraphics[width=\textwidth]{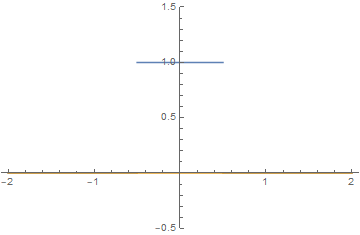}
  \end{minipage}
	\hfill
  \begin{minipage}[b]{0.19\textwidth}
    \includegraphics[width=\textwidth]{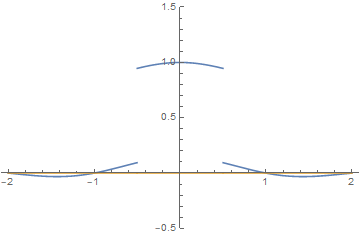}
  \end{minipage}
	\hfill
  \begin{minipage}[b]{0.19\textwidth}
    \includegraphics[width=\textwidth]{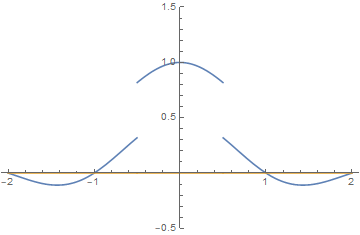}
  \end{minipage}
	\hfill
  \begin{minipage}[b]{0.19\textwidth}
    \includegraphics[width=\textwidth]{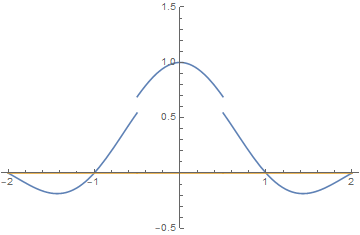}
  \end{minipage}
	\hfill
  \begin{minipage}[b]{0.19\textwidth}
    \includegraphics[width=\textwidth]{FrFTfull.png}
  \end{minipage}
\caption{The alternate fractional Fourier transform, $\mathcal{F}^\alpha[\chi_{[-1/2,1/2]}]$ for $\alpha=0,\frac{1}{4},\frac{1}{2},\frac{3}{4},1$. The blue curve represents the real part, while the orange curve represents the imaginary part.}
\end{figure}

\begin{figure}[h!]
  \centering
  \begin{minipage}[b]{0.19\textwidth}
    \includegraphics[width=\textwidth]{bFrFTnone.png}
  \end{minipage}
	\hfill
  \begin{minipage}[b]{0.19\textwidth}
    \includegraphics[width=\textwidth]{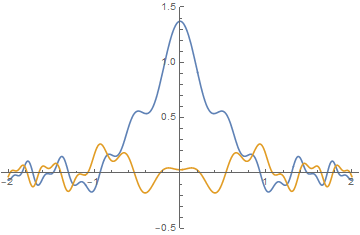}
  \end{minipage}
	\hfill
  \begin{minipage}[b]{0.19\textwidth}
    \includegraphics[width=\textwidth]{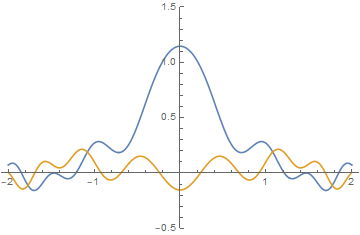}
  \end{minipage}
	\hfill
  \begin{minipage}[b]{0.19\textwidth}
    \includegraphics[width=\textwidth]{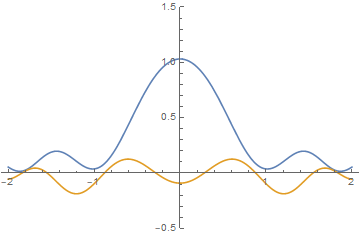}
  \end{minipage}
	\hfill
  \begin{minipage}[b]{0.19\textwidth}
    \includegraphics[width=\textwidth]{FrFTfull.png}
  \end{minipage}
\caption{The literature fractional Fourier transform, $\mathcal{F}_\alpha[\chi_{[-1/2,1/2]}]$ for $\alpha=0,\frac{\pi}{8},\frac{\pi}{4},\frac{3\pi}{8},\frac{\pi}{2}$. The blue curve represents the real part, while the orange curve represents the imaginary part.}
\end{figure}
Indeed, it is quickly seen that the literature fractional Fourier transform has smoothing properties for all $\alpha\notin\Z$, while the alternate fractional Fourier transform has smoothing properties for only some $\alpha\in\Z$. However, there remains an elegance to the alternate fractional Fourier transform that is most readily seen for small values of $\alpha$. Observe the following figure, comparing the $\frac{1}{50}^{th}$ power of the Fourier operators.
\newpage
\begin{figure}[h!]
  \centering
  \begin{minipage}[b]{0.49\textwidth}
    \includegraphics[width=\textwidth]{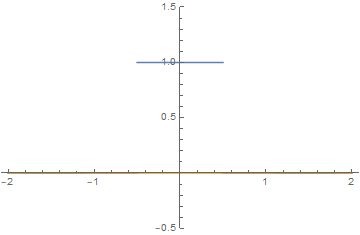}
  \end{minipage}
	\hfill
  \begin{minipage}[b]{0.49\textwidth}
    \includegraphics[width=\textwidth]{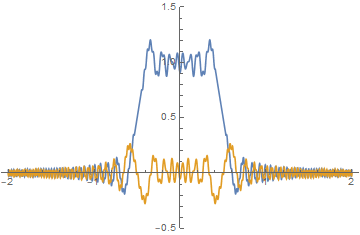}
  \end{minipage}
\caption{The alternate fractional Fourier transform $\mathcal{F}^{\frac{1}{50}}[\chi_{[-1/2,1/2]}]$ on the left, compared to the literature fractional Fourier transform, $\mathcal{F}_{\frac{\pi}{100}}[\chi_{[-1/2,1/2]}]$ on the right. The blue curve represents the real part, while the orange curve represents the imaginary part.}
\end{figure}

Because of the smoothing properties of the literature fractional Fourier transform, continuity of the operator power forces that non-local frequencies tend to infinity as the power tends to zero, as the image needs to approximate a non-smooth function by a smooth one. Since the alternate fractional Fourier transform is simply a linear combination of a function, its Fourier transform, and an involution of both, it maintains the locality of smoothness.

\subsection{Bad cases}

\begin{Remark}
	Let $T_t:L^p(\R)\to L^p(\R)$ be the forward translation operator such that $T_t[f](x)=f(x-t)$. Then
	\eqn{T_t\neq (T_1)^t=\sum_{n\in\Z}\sinc(t-n)(T_1)^n=\sum_{n\in\Z}\sinc(t-n)T_n.}
\end{Remark}

\begin{proof}
	Consider the function $\chi_{[0,1/2]}\in L^p(\R)$. Then $T_n[\chi_{[0,1/2]}](x)=\chi_{[0,1/2]}(x-n)=\chi_{[n,n+1/2]}(x)$. Hence
	\eqn{\text{supp }(T_1)^t[\chi_{[0,1/2]}](x)\cap(1/2,1)=\emptyset}
	for all $t\in\R$. However, it is simple to see that
	\eqn{\text{supp }T_{1/2}[\chi_{[0,1/2]}](x)\cap(1/2,1)=(1/2,1).}
	Therefore $T_t\neq(T_1)^t$.
\end{proof}

\begin{Proposition}\label{P:Translation}
	Let $T_t:L^p(\R)\to L^p(\R)$ be the forward translation operator such that $T_t[f](x)=f(x-t)$. Then
	\eqn{T_t= \lim_{k\to\infty}(T_{1/k})^{kt}=\lim_{k\to\infty}\sum_{n\in\Z}\sinc(kt-n)(T_{1/k})^n=\lim_{k\to\infty}\sum_{n\in\Z}\sinc(kt-n)T_{n/k}.}
\end{Proposition}

\begin{proof}
	\eqn{\lim_{k\to\infty}\sum_{n\in\Z}\sinc(kt-n)T_{n/k}=\int_\R \delta(x-t)T_x dx=T_t}
	by the definition of a Bochner integral as a uniform limit of operator-valued Riemann sums.
\end{proof}

\begin{Proposition}\label{P:Semigroup}
	Let $e^{A t}$ be a unitary $C_0$-semigroup. Then
	\eqn{e^{At}= \lim_{k\to\infty}(e^{\frac{A}{k}})^{kt}=\lim_{k\to\infty}\sum_{n\in\Z}\sinc(kt-n)(e^{\frac{A}{k}})^n=\lim_{k\to\infty}\sum_{n\in\Z}\sinc(kt-n)e^{A\frac{n}{k}}.}
\end{Proposition}

\begin{proof}
	\eqn{\lim_{k\to\infty}\sum_{n\in\Z}\sinc(kt-n)e^{A\frac{n}{k}}=\int_\R \delta(x-t)e^{Ax} dx=e^{At}}
	by the definition of a Bochner integral as a uniform limit of operator-valued Riemann sums.
\end{proof}

\section{Conclusion and further research}

Though not defined for the entire class of operator which the functional calculus allows, analytic interpolation of the integer powers of operators can produce fractional powers of operators which agree with that of the functional calculus in many cases. Such interpolations allow one to approximate fractional powers by finite sums of integer powers of operators, and thus may find much use in numerical analysis of fractional PDEs and the time-frequency analysis of the fractional Fourier transform. Further, these results provide an interpolation procedure for the Riemann zeta function. It is hoped that these results may shine light on a solution to the persistently elusive Riemann hypothesis.

\end{document}